\documentclass[a4paper,11pt]{amsart}

       \usepackage{palatino, verbatim}
        \usepackage[latin1]{inputenc}
        \usepackage[T1]{fontenc}
        \usepackage{amsthm, amsfonts}
        \usepackage{amsfonts}
        \usepackage{graphicx}
        \usepackage{amssymb}
        \usepackage{amsmath}
        \usepackage{latexsym}
        \usepackage{color}
      
        \usepackage[all]{xy}

        \newtheorem{thm}{Theorem}[section]
          \newtheorem{cor}[thm]{Corollary}
          \newtheorem{lem}[thm]{Lemma}
          \newtheorem{prop}[thm]{Proposition}

        \theoremstyle{definition}
          
          \newtheorem{rem}{Remark}
          \newtheorem{ex}{Example}

          \newcommand\M{{\mathcal M}}

          \newcommand\e{{\mathfrak e}}
           \newcommand\f{{\mathfrak f}}
          
          \newcommand\ZZ{{\mathbb Z}}

          \newcommand\Pic{\mathrm{Pic}}

\begin{document}
\title{Genus two curves on Abelian surfaces}
\author[A.~L.~Knutsen]{Andreas Leopold Knutsen}
\address{Andreas Leopold Knutsen, Department of Mathematics, University of Bergen,
Postboks 7800,
5020 Bergen, Norway}
\email{andreas.knutsen@math.uib.no}

\author[M.~Lelli-Chiesa]{Margherita Lelli-Chiesa}
\address{Università degli studi Roma Tre, Dipartimento di Matematica e Fisica, Largo San Leonardo Murialdo 1, 00146 Roma}
\email{margherita.lellichiesa@uniroma3.it}

\begin{abstract}
This paper deals with singularities of genus $2$ curves on a general $(d_1,d_2)$-polarized abelian surface $(S,L)$. In analogy with Chen's results concerning rational curves on K3 surfaces \cite{ch1,ch2}, it is natural to ask whether all such curves are nodal. We prove that this holds true if and only if $d_2$ is not divisible by $4$. In the cases where $d_2$ is a multiple of $4$, we exhibit genus $2$ curves in $|L|$ that have a triple, $4$-tuple or $6$-tuple point. We show that these are the only possible types of unnodal singularities of a genus $2$ curve in $|L|$. Furthermore, with no assumption on $d_1$ and $d_2$, we prove the existence of at least one nodal genus $2$ curve in $|L|$. As a corollary, we obtain nonemptiness of all Severi varieties on general abelian surfaces and hence generalize \cite[Thm. 1.1]{KLM} to nonprimitive polarizations.
\end{abstract}

\maketitle

\section{Introduction}
 The minimal geometric genus of any curve lying on a general abelian surface is $2$ and there are finitely many curves of such genus in a fixed linear system. The role of genus two curves on abelian surfaces is thus analogous to that of rational curves on $K3$ surfaces, but until now it has not been investigated as extensively. Their enumeration is now well understood. Their count in the primitive case was carried out by G\"ottsche \cite{Go}, Debarre \cite{De} and Lange-Sernesi \cite{LS1}, and used in \cite{De} in order to compute the Euler characteristic of generalized Kummer varieties. Only recently, Bryan, Oberdieck, Pandharipande and Yin \cite{BOPY} handled the nonprimitive case, thus obtaining a formula parallel to the full Yau-Zaslow conjecture for rational curves on K3 surfaces (cf. \cite{KMPS}). 
 
Singularities of rational curves on $K3$ surfaces have received plenty of attention. Mumford \cite[Appendix]{MM} first proved the existence of a nodal rational curve in the primitive linear system $|L|$ on a general genus $g$ polarized $K3$ surface $(S,L)$; as a byproduct, he obtained nonemptiness of the Severi variety $|L|_\delta$ parametrizing $\delta$-nodal curves in $|L|$ for any integer $0\leq \delta\leq g$. His results were then generalized by Chen \cite{ch1,ch2} to nonprimitive linear systems. In the primitive case, Chen managed to deal with all rational curves in $|L|$ showing that they are all nodal; the analogue for nonprimitive linear systems is still an open problem. 

Singularities of genus $2$ curves on abelian surfaces are not as well understood, even though they are necessarily ordinary (cf. \cite[Prop. 2.2]{LS2}). The natural question whether any genus $2$ curve on a general $(d_1,d_2)$-polarized abelian surface is nodal \cite[Pb. 2.7]{LC} has negative answer if one does not make any assumption on $d_1$ and $d_2$. Indeed, multiplication by $2$ on a principally polarized abelian surface $(A,L)$ identifies the six Weierstrass points of its theta divisor, whose image is thus a genus $2$ curve with  a $6$-tuple point lying in (a translate) of the linear system $|L^{\otimes 4}|$ (cf. Example \ref{sestuplo}). Since this is a polarization of type $(4,4)$, all genus $2$ curves may still be expected to be nodal  in primitive linear systems (or even in linear systems not divisible by $4$, cf. \cite[Conj. 2.10]{LC}). Our main result is that this expectation does not hold in its full generality and detects all the cases where it fails, thus completely answering the question.

\begin{thm}\label{genus two}
Let $(S,L)$ be a general abelian surface with a polarization of type $(d_1,d_2)$. Then any genus $2$ curve in the linear system $|L|$ is nodal if and only if $4$ does not divide $d_2$.
\end{thm} 
When $d_2$ is a multiple of $4$, we exhibit genus $2$ curves in $|L|$ that have an unnodal singularity and, more precisely, a triple, a $4$-tuple or a $6$-tuple point (cf. Examples \ref{quadruplo} and \ref{sestuplo}). We also show that these are the only types of unnodal singularities that a genus $2$ curve on a general abelian surface may acquire (cf. Remark \ref{singu}). To our knowledge, the best bound on the order of such a singularity in the literature was $\frac{1}{2}\left( 1+\sqrt{8d_1d_2-7}\right)$ by Lange-Sernesi, cf. \cite[Prop. 2.2]{LS2}. The existence of unnodal genus $2$ curves in all primitive linear systems of type $(1,4k)$ is quite striking and highlights a major difference with the $K3$ case.

When $4$ divides $d_2$, the above theorem does not exclude the existence in $|L|$ of some nodal genus $2$ curves. This is indeed proved  in the following:

\begin{thm}\label{nodal}
Let $(S,L)$ be a general $(d_1,d_2)$-polarized abelian surface. Then the linear system $|L|$ contains a nodal curve of genus $2$.
\end{thm}

Given a nodal genus $2$ curve as above, standard deformation theory enables one to smooth any of its nodes independently remaining inside the linear system $|L|$. As a consequence, Theorem \ref{nodal} yields nonemptiness of all Severi varieties on general abelian surfaces:
\begin{cor}\label{severi}
Let $(S,L)$ be a general $(d_1,d_2)$-polarized abelian surface. Then, for any $0\leq \delta\leq d_1d_2-1$ the Severi variety $|L|_{\delta}$ is nonempty and smooth of  dimension equal to $d_1d_2-1-\delta$.
\end{cor}
This generalizes \cite[Thm. 1.1]{KLM} to nonprimitive linear systems. Note that, since $S$ has trivial canonical bundle, the regularity of $|L|_\delta$ stated in Corollary \ref{severi} follows for free from its nonemptiness  by the proofs of Propositions 1.1 and 1.2 in \cite{LS2}. We mention that the irreducible components of the Severi varieties on a general primitively polarized abelian surface have been determined very recently by Zahariuc in \cite{Za}.

We now spend some words on the proofs of Theorems \ref{genus two} and \ref{nodal}. In contrast to the methods proposed  in \cite{ch1,ch2,KLM}, we need neither to degenerate $S$ to a singular surface nor to specialize it to an abelian surface with large Neron-Severi group. Instead, we exploit the universal property of Jacobians in order to translate the if part of Theorem \ref{genus two} and Theorem \ref{nodal} into the following statement concerning Brill-Noether theory on a general curve of genus $2$:
\begin{thm}\label{thm:plane}
Let $[C]\in \M_2$ be a general genus $2$ curve and fix any integer $d\geq 4$.  If $C$ admits a $g^2_d$ totally ramified at three points $P_1,P_2,P_3$, then $d$ is even and $P_1,P_2,P_3$ are Weierstrass points.
\end{thm}
We refer to Section \S \ref{prima} for the details of this reduction, that we mention here only briefly. The key fact is that any genus $2$ curve $\overline C$ on a $(d_1,d_2)$-polarized abelian surface $S$ with normalization $C$ arises as image of a composition 
\begin{equation}\label{compo}
C\stackrel{u}{\hookrightarrow} J(C)\stackrel{\lambda}{\rightarrow} S,
\end{equation}
where $u$ is the Abel-Jacobi map and $\lambda$ is an isogeny. Three points $P_1,P_2,P_3\in C=u(C)$ identified by $\lambda$ necessarily differ by elements in its kernel. Since the order of any such element is divisible by $d_2$, the three divisors $d_2P_1,d_2P_2,d_2P_3\in C_d$ are linearly equivalent and thus define (for $d_2\geq 4$) a $g^2_{d_2}$ on $C$ totally ramified at three points. Theorem \ref{thm:plane} excludes the existence of such a linear series for $C$ general and odd values of $d_2$, thus implying our main results in these cases.  If instead $d_2$ is even, a $g^2_{d_2}$ totally ramified at three points does exist: as soon as $P_1,P_2,P_3$ are Weierstrass points of $C$, the divisors $2P_1,2P_2,2P_3$ are linearly equivalent and thus the same holds true for $d_2P_1,d_2P_2,d_2P_3$. Conversely, by Theorem \ref{thm:plane}, any $g^2_{d_2}$ with three points of total ramification on $C$ is of this type. This characterization is used in Section \S \ref{prima} both to prove the if part of Theorem \ref{genus two} and Theorem \ref{nodal} for $d_2\equiv\,2\,\mathrm{mod}\,4$, and to provide examples of genus $2$ curves with a triple, $4$-tuple or $6$-tuple point (cf. Examples \ref{quadruplo} and \ref{sestuplo}) when $d_2\equiv\,0\,\mathrm{mod}\,4$ implying the only if part of Theorem \ref{genus two}. These examples are based on the construction of suitable isogenies $\lambda$ as in \eqref{compo} or, equivalently by taking their kernels, suitable  isotropic (with respect to the commutator pairing) subgroups of the group $J(C)[d_1d_2]$ of $d_1d_2$-torsion points of $J(C)$. 

Section \S \ref{seconda} is devoted to the proof of Theorem \ref{thm:plane}. This is done in two steps. First, we degenerate $C$ to the transversal union of two elliptic curves meeting at a point and reduce Theorem \ref{thm:plane} into a statement of Brill-Noether theory with ramification on a general elliptic curve (cf. Proposition \ref{elliptic}). This reduction seriously involves the theory of limit linear series on curves of compact type, for which we refer to the original papers by Eisenbud and Harris \cite{EH1,EH2,EH3}. Proposition \ref{elliptic} is then proved by an infinitesimal study of a {\em generalized Severi variety} (cf. \cite{CH,Za2}).

\textbf{Acknowledgements:} We are especially grateful to Nicolò Sibilla for numerous valuable conversations on the topic and to Alessandro D'Andrea for his substantial help in the proof of Lemma \ref{group}. We have benefited from interesting correspondence with Igor Dolgachev. The first named author has been partially supported by grant n. 261756 of the Research Council of Norway and by the Trond Mohn Foundation.

\section{Polarized isogenies and proof of the main theorems}\label{prima}
In this section we review some known facts concerning polarized isogenies and genus $2$ curves on complex abelian surfaces and reduce the proof of Theorems \ref{genus two} and \ref{nodal} to a statement concerning Brill-Noether theory with prescribed ramification on a general curve of genus $2$.  

\subsection{Polarized isogenies and genus $2$ curves}Let $S$ be an abelian surface defined over $\mathbb{C}$ and consider a genus $2$ curve $\overline{C}\subset S$ such that the line bundle $L:=\mathcal{O}_S(\overline{C})$ is a polarization of type $(d_1,d_2)$. The normalization map $\nu:C\to \overline C\subset S$ then factors as
$$
C\stackrel{u}{\hookrightarrow} J(C)\stackrel{\lambda}{\rightarrow} S,
$$
where $u$ is the Abel-Jacobi map (that is is an embedding only defined up to translation) and $\lambda$ is an isogeny. We set $A:=J(C)$.  

By the Push-Pull formula, the above isogeny $\lambda$ has degree $d_1d_2$ and thus $\lambda^* L\simeq L_1^{\otimes d_1d_2}$, where $L_1$ is a principal polarization on $A$. We write $A=V/\Lambda$, where $V$ is a $2$-dimensional $\mathbb{C}$-vector space and $\Lambda$ is a rank $4$ lattice. Chosen a symplectic basis $\lambda_1,\lambda_2,\mu_1,\mu_2$ of $\Lambda$, we denote by  $\mathfrak e_1':=\lambda_1/(d_1d_2)$, $\mathfrak e_2':=\lambda_2/(d_1d_2)$, $\mathfrak f_1':=\mu_1/(d_1d_2)$, $\mathfrak f_2':=\mu_2/(d_1d_2)$ the standard generators of the group $A[d_1d_2]$ of $d_1d_2$-torsion points of $A$. By definition, the commutator pairing on $A[d_1d_2]$ is the nondegenerate multiplicative alternating form
$$
e_{d_1d_2}:A[d_1d_2]\times A[d_1d_2]\to \mathbb{C}^*
$$
that takes value $1$ on all pairs of standard generators with the only two following exceptions:
\begin{equation}\label{pair}
e_{d_1d_2}(\mathfrak e_1',\mathfrak f_1')=e_{d_1d_2}(\mathfrak e_2',\mathfrak f_2')=e^{\frac{2\pi i}{d_1d_2}}.
\end{equation}
For a fixed principally polarized abelian surface $A$ with a fixed theta divisor $\Theta$, by \cite[\S 23]{Mu} there is a bijection between the following two sets:
\begin{itemize}
\item[(*)] polarized isogenies $\lambda:A \to S$ onto abelian surfaces $S$ such that $\lambda(\Theta)\in |L|$ for some polarization $L$ on $S$ of type $(d_1,d_2)$;
\item[(**)] isotropic subgroups $G$ of $A[d_1d_2]$ of cardinality $d_1d_2$ such that $G^{\perp}/G\simeq \mathbb{Z}_{d_1}^{\oplus 2}\oplus\mathbb{Z}_{d_2}^{\oplus 2}$.
\end{itemize}
Indeed, the kernel $G$ of any isogeny $\lambda$ in (*) is a subgroup  of $A[d_1d_2]$ in (**); furthermore, $G^{\perp}/G$ is isomorphic to the kernel $K(L)$ of the isogeny  defined by $L$:
\begin{equation*}
\phi_L:S\rightarrow \hat S,\,\,\,\, \phi_L(x)=t_x^*L\otimes L^\vee,
\end{equation*}
where $t_x$ denotes the translation by $x$ on $S$. Viceversa, given a subgroup $G$ in (**), the quotient map $\lambda:A\to A/G$ is an isogeny as in (*).

Given $\lambda$ as in (*), let
$$
\hat\lambda: \hat S\to \hat A
$$
be its dual isogeny and denote by $\hat L$ the dual polarization of $L$. Again by \cite[\S 23]{Mu} the kernel $\hat G$ of $\hat\lambda$ is a maximal isotropic subgroup of $K(\hat L)\simeq \mathbb Z_{d_1}^{\oplus 2}\oplus\mathbb Z_{d_2}^{\oplus 2}$. On the other hand, $\hat G$ is the character group of $G$ (cf. \cite[Prop. 2.4.3]{BL}) and thus $\hat G\simeq G$. In particular, the order of any element of $G=\ker \lambda$ divides $d_2$.

\subsection{Reduction of Theorem \ref{genus two} to Theorem \ref{thm:plane} for odd values of $d_2$} Since isogenies are \'etale, all singularities of the image $\lambda(\Theta)$ of a theta divisor  under an isogeny $\lambda$ as in (*) are ordinary (cf. \cite[Prop. 2.2]{LS2}). The only pathology that might prevent $\lambda(\Theta)$ from being nodal is thus  the existence of three points $x,y,z\in \Theta$ such that $\lambda(x)=\lambda(y)=\lambda(z)$. Since the order of any element in the kernel of $\lambda$ divides $d_2$, such a triple of points $x,y,z$ would be identified by the multiplication map
$$
m_{d_2}:A\to A;
$$
equivalently, if $A=J(C)$,  the three divisors $d_2 x,d_2 y,d_2 z$ on $C$ would be linearly equivalent. For $d_2\geq 4$ this implies the existence of a $g^2_{d_2}$ on $C$ totally ramified at $x,y,z$. In the cases $d_2=2$ and $d_2=3$ the same conclusion holds up to replacing $d_2$ with a multiple of it. Theorem \ref{genus two} for odd values of $d_2$ then follows from Theorem \ref{thm:plane}.

\subsection{Theorem \ref{genus two} for even values of $d_2$} Theorem \ref{thm:plane} also implies that the image of a theta divisor under an isogeny $\lambda$ as in (*) for even values of $d_2$ may have non-nodal singularities only at the image of its Weierstrass points. In order to analize this possibility, we denote by $\mathfrak e_1:=\lambda_1/2$, $\mathfrak e_2:=\lambda_2/2$, $\mathfrak f_1:=\mu_1/2$, $\mathfrak f_2:=\mu_2/2$ the standard generators of the group $A[2]$ of $2$-torsion points  of $A$. Note that  
\begin{equation}\label{easy}
\e_i=\frac{d_1d_2}{2}\e_i',\,\,\,\f_i=\frac{d_1d_2}{2}\f_i'\textrm{ for } i\in\{1,2\}.
\end{equation}  
As the Abel-Jacobi map $u:C{\hookrightarrow} J(C)=A$ is defined up to translation, we may assume its image to coincide with a symmetric theta divisor $\Theta$ so that (the image under $u$ of) the six Weierstrass points of $C$ are exactly the $2$-torsion points of $A$ contained in $\Theta$, namely, without loss of generality by \cite[Ex. 10.2.7]{BL}, the points:
\begin{equation}\label{weie}
\mathfrak e_1,\mathfrak f_1,\mathfrak e_1+\mathfrak f_1,\mathfrak e_2,\mathfrak f_2,\mathfrak e_2+\mathfrak f_2.
\end{equation}

Theorem \ref{genus two} for $d_2\equiv  2\, \mathrm{mod}\,4$ then follows from the following Lemma.

\begin{lem}\label{group}
Let $G$ be an isotropic subgroup of $A[d_1d_2]$ such that $|G|=d_1d_2$, $G^{\perp}/G\simeq \mathbb{Z}_{d_1}^{\oplus 2}\oplus\mathbb{Z}_{d_2}^{\oplus 2}$ and at least three among the six $2$-torsion points in \eqref{weie} lie in the same $G$-orbit of $A[d_1d_2]$. Then, one necessarily has $d_2\equiv 0\, \mathrm{mod}\,4$.
\end{lem}
\begin{proof}
Up to exchanging the $\e_i$'s with the $\f_i$'s and up to relabelling the indices $i$, we may assume that one of the three points in the same $G$-orbit of $A[d_1d_2]$ is $\e_1$. It follows that $G$ contains at least two elements $g_1,g_2$ in the following set:
\begin{equation}\label{set}
\left\{\e_1+\f_1,\f_1,\e_2+\e_1,\f_2+\e_1,\e_2+\f_2+\e_1\right\}.
\end{equation}

Using \eqref{pair} and \eqref{easy} one obtains $e_{d_1d_2}(g_1,g_2)=e^{\frac{2d_1d_2\pi i}{4}}$ (up to exchanging $g_1$ and $g_2$), and thus 
\begin{equation}\label{film}
d_1d_2\equiv 0\, \mathrm{mod}\, 4
\end{equation} since $G$ is isotropic. 

In order to exclude the case $d_2\equiv 2\, \mathrm{mod}\, 4$ (that would also imply $d_1\equiv 2\, \mathrm{mod}\, 4$ by \eqref{film}), we proceed by contradiction.  Both the elements $\frac{2g_1}{d_1d_2}$ and $\frac{2g_2}{d_1d_2}$ have order $d_1d_2$ and $e_{d_1d_2}(\frac{2g_1}{d_1d_2},\frac{2g_2}{d_1d_2})=e^{\frac{2\pi i}{d_1d_2}}$; therefore, there exists an automorphism $\varphi$ of $A[d_1d_2]$ preserving the alternating form $e_{d_1d_2}$ such that $\varphi(\frac{2g_1}{d_1d_2})=\e_1'$ and $\varphi(\frac{2g_2}{d_1d_2})=\e_2'$. In particular, we may assume $g_1= \e_1$ and $g_2=\f_1$. We consider the group 
$$K:=\langle  \e_1,\f_1\rangle,$$
and its orthogonal
$K^\perp=\langle 2\e_1',2\f_1',\e_2',\f_2'\rangle$. By \eqref{easy}, one gets
\begin{equation}\label{kappa}
K^\perp/K\simeq \ZZ_{\frac{d_1d_2}{4}}^{\oplus 2}\oplus\ZZ_{d_1d_2}^{\oplus 2}\simeq \ZZ_{\frac{d_1d_2}{4}}^{\oplus 4}\oplus\ZZ_{4}^{\oplus 2}\textrm{ with }\frac{d_1d_2}{4} \textrm{ odd,}
\end{equation}
where the second isomorphism follows from the assumption $d_2\equiv 2\, \mathrm{mod}\, 4$. The inclusions $K<G<G^\perp<K^\perp$ imply that 
\begin{equation*}
K^\perp/K>G^\perp/K\textrm{ and } G^\perp/G\simeq (G^\perp/K)/(G/K);
\end{equation*}
in particular $G^\perp/G$ is a quotient of a subgroup of $K^\perp/K$. However, our assumption yields$$
G^\perp/G\simeq \ZZ_{d_1}^{\oplus 2}\oplus \ZZ_{d_2}^{\oplus 2}\simeq \ZZ_2^{\oplus 4}\oplus\ZZ_{\frac{d_1}{2}}^{\oplus 2}\oplus \ZZ_{\frac{d_2}{2}}^{\oplus 2} \textrm{ with }\frac{d_1}{2},\frac{d_2}{2} \textrm{ odd}.$$
As a consequence, $\ZZ_2^{\oplus 4}$ is a quotient of a subgroup of $K^\perp/K$ and thus of $\ZZ_4^{\oplus 2}$ by \eqref{kappa}. This is a contradiction because the only quotient of a subgroup of $\ZZ_4^{\oplus 2}$ having cardinality $16$ is $\ZZ_4^{\oplus 2}$ itself.
\end{proof}
By the following example, as soon as $d_2\equiv 0\,\mathrm{mod}\, 4$, there do exist isotropic subgroups $G$ of $A[d_1d_2]$ as in Lemma \ref{group}. As a consequence,  a general polarized abelian surface of type $(d_1,d_2)$ contains an unnodal genus $2$ curve and the only if part of Theorem \ref{genus two} follows.

\begin{ex}\label{quadruplo}
We fix positive integers $d_1,d_2,a,b$ such that $d_1|d_2$ and the relation $ab=d_1^2d_2$ holds. We consider the following subgroup of $A[d_1d_2]$
$$G:=\langle a\mathfrak e_1',b\mathfrak f_1',d_2\mathfrak e_2'\rangle.
$$
One has $|G|=\frac{d_1d_2}{a}\cdot\frac{d_1d_2}{b}\cdot d_1=d_1d_2$ and 
$$G^{\perp}=\left\langle \frac{d_1d_2}{b}\mathfrak e_1', \frac{d_1d_2}{a}\mathfrak f_1', \mathfrak e_2', d_1\mathfrak f_2'\right\rangle,
$$
and hence $G^{\perp}/G\simeq \mathbb{Z}_{d_1}^{\oplus2}\oplus \mathbb{Z}_{d_2}^{\oplus 2}$. In particular, the group $G$ corresponds to a polarized isogeny from the principally polarized abelian surface $A$ to a $(d_1,d_2)$-polarized abelian surface $(S,L)$. If both $a$ and $b$ divide $d_1d_2/2$ (and thus $ab=d_1^2d_2$ divides $(d_1d_2)^2/4$, or equivalently, $d_2\equiv 0\,\mathrm{mod}\, 4$), then $G$ contains the $2$-torsion points $\mathfrak e_1,\mathfrak f_1,\mathfrak e_1+\mathfrak f_1$ and we find a genus $2$ curve in $|L|$ with a singularity that is (at least) a triple point.  Note that for any values of $d_1,d_2$ such that $d_1|d_2$ and $d_2\equiv 0\,\mathrm{mod}\, 4$, the integers $a=d_1d_2/2$ and $b=2d_1$ satisfy all the above conditions. 
If moreover $d_1$ is even, then $G$ contains also the point $\mathfrak e_2$ and the image of the theta divisor aquires a $4$-tuple point. 

\end{ex}

To our knowledge the only example  of an unnodal genus $2$ curve on an abelian surface from the published literature until now was the image of the theta divisor under the multiplication by two on a principally polarized abelian surface. This example has played interesting roles in 
various works concerning curve singularities (cf. \cite[Ex. 4.14]{DS}) and Seshadri constants (cf. \cite[Pf. of Prop. 2]{St}, \cite[Rmk. 6.3]{Ba}, \cite[Ex. 4.2]{KSS}). We recall and generalize this example:

\begin{ex}\label{sestuplo}
Assume we have an isogeny $\lambda$ as in (*) identifiying all the six Weierstrass points of $\Theta$. The group $A[2]$ of $2$-torsion points of $A$ is necessarily contained in the kernel of such a $\lambda$, that hence factors through the multiplication  by $2$
$$m_2:A\to A.$$  In fact, the image $m_2(\Theta)$ has only one singularity at the image of the six Weierstrass points, that is thus a $6$-tuple point. Furthermore, $m_2(\Theta)\in |L_1^{\otimes 4}|$ where $L_1$ is a principal polarization on $A$. As soon as $d_1\equiv 0\,\mathrm{mod}\,4$, one constructs a genus $2$ curve with a sixtuple point  on a general $(d_1,d_2)$- polarized abelian surface $(S,L)$ by composing $m_2$ with an isogeny $\lambda':A\to S$ such that $\lambda'(\Theta)\in |L'|$ where $L'$ is a polarization on $S$ satisfying $L'^{\otimes 4}\simeq L$. 
\end{ex}

\begin{rem}\label{singu}
Theorem \ref{thm:plane} along with the fact that any smooth genus $2$ curve has exactly $6$ Weierstrass points yields that  $6$ is the maximal order of any singularity of a genus $2$ curve on a general abelian surface. Examples \ref{quadruplo} and \ref{sestuplo} exhibit genus $2$ curves with a triple, a $4$-tuple or a $6$-tuple point. It is natural to ask whether one can construct a genus $2$ curve with a $5$-tuple point. Such a curve would correspond to an isotropic subgroup $G$ containing exactly $4$ points in the set \eqref{set}. However, any subgroup of $A[2]$ generated by $4$ elements in the set \eqref{set} contains $\e_1,\e_2,\f_1,\f_2$ and thus coincides with $A[2]$. As a consequence, if one requires $G$ to contain four points in \eqref{set}, then $G$ contains the whole $A[2]$ and one falls in Example \ref{sestuplo} thus obtaining a $6$-tuple and not a $5$-tuple point.
\end{rem}

\begin{rem}
While looking for a proof of Theorem \ref{genus two}, we realized that the proof of  \cite[Proposition 3.1]{DL} contains a gap since it somehow assumes that  an isogeny between two principally polarized abelian  surfaces $\lambda:A\to B$ never identifies three or more points on the theta divisor of $A$. In \cite{DL} the abelian varieties are defined over an algebraically closed field $\mathbb{K}$ of arbitrary characteristic and the kernel of $\lambda$ is a maximal isotropic subgroup of $A[p]$ for some prime integer $p\neq\,\mathrm{char}\,\mathbb{K}$. Theorem \ref{thm:plane} repairs the mentioned gap for $\mathbb{K}=\mathbb{C}$.
\end{rem}

\subsection{Reduction of Theorem \ref{nodal} to Theorem \ref{thm:plane}}We conclude this section by proving the following lemma, to which Theorem \ref{nodal} reduces thanks to Theorem \ref{thm:plane}.
\begin{lem}
For any positive integers $d_1,d_2$ with $d_1|d_2$, there exists an isotropic  subgroup $G$ of $A[d_1d_2]$ as in (**) such that any $G$-orbit of $A[d_1d_2]$ contains at most two points in \eqref{weie}.
\end{lem}
\begin{proof}
The group
$$
G:=\langle d_1 \mathfrak e_1', d_2\mathfrak e_2'\rangle
$$
is clearly isotropic. By \eqref{easy}, $G$ contains $\e_1$ if $d_2$ is even and $\e_2$ if $d_1$ is even;  in no case $G$ contains other elements of order $2$.  As a consequence, the only set of points contained in \eqref{weie} and lying in the same $G$-orbit are $\{\mathfrak f_1,\e_1+\f_1\}$ for even values of $d_2$, and $\{\mathfrak f_2,\e_2+\f_2\}$ if $d_1$ is even. One easily checks that
$$
G^\perp:=\langle \mathfrak e_1', d_2\mathfrak  f_1',  \mathfrak e_2', d_1\mathfrak  f_2'\rangle
$$
so that $G^{\perp}/G\simeq \mathbb{Z}_{d_1}^{\oplus 2}\oplus\mathbb{Z}_{d_2}^{\oplus 2}$.

\end{proof}

\section{Proof of Theorem \ref{thm:plane}}\label{seconda}

We proceed by degeneration to a curve $C_0$ having two irreducible smooth elliptic components $E_1$ and $E_2$ meeting at a point $P$. 

Let $\pi:\mathcal{C}\to B$ be a flat family of curves over a {\em local} one-dimensional base $B$ (that is, $B=\mathrm{Spec}R$ for some discrete valuation ring $R$) with special fiber $C_0$ and generic fiber $C_b$ being a smooth irreducible curve of genus $2$; also assume that the total space $\mathcal{C}$ is smooth. A relative linear series of type $g^2_d$ on $\mathcal{C}$ is a pair $\mathfrak{l}=(\mathcal{A},\mathcal{V})$ such that $\mathcal{A}$ is a line bundle on $\mathcal{C}$ flat over $B$ and $\mathcal{V}$ is a rank-$3$ subbundle of $\pi_*\mathcal{A}$. We assume the existence  of a linear series $l_b=(A_b,V_b)$ of type $g^2_d$ on the generic fiber $C_b$ of $\pi$ totally ramified at three points. Possibly after finitely many sequences of base changes and blow-ups at the nodes of the special fiber, we obtain a family $\pi':\mathcal{C'}\to B'$ such that:
\begin{itemize}
\item[(i)] the generic fiber of $\pi'$ is again $C_b$;
\item[(ii)]  the special fiber $C_0'$ of $\pi'$ is obtained from $C_0$ by inserting a chain of $h\geq 0$ rational curves at the node $P$;
\item[(iii)] $l_b$ is the restriction of a relative linear series $\mathfrak{l}=(\mathcal{A},\mathcal{V})$ on $\mathcal{C}'$;
\item[(iv)] there are three sections  $\sigma_1,\sigma_2,\sigma_3$ of $\pi'$ such that $l_b$ is totally ramified at the points $\sigma_1(b),\sigma_2(b),\sigma_3(b)$; 
\item[(v)] the points  $x_1:=\sigma_1(0)$, $x_2:=\sigma_2(0)$, $x_3:=\sigma_3(0)$ lie in the smooth locus of $C_0'$ (but are allowed to coincide).
\end{itemize}
We label the rational components inserted at $P$ with $\gamma_1, \ldots, \gamma_h$ and set $\gamma_0:=E_1$, $\gamma_{h+1}:=E_2$ and $P_i:=\gamma_{i-1}\cap \gamma_i$ for $1\leq i\leq h+1$. The restriction of $\mathfrak{l}$ to $C_0'$ is a (crude) limit linear series \cite{EH2}, whose {\em aspect} on $\gamma_i$ (cf. \cite[Def. p. 348]{EH2}) is denoted by $l_i=(A_i,V_i)$. 
If $P_j\in \gamma_i$, let $\underline{\alpha}^i(P_j)=(\alpha^i_0(P_j),\alpha^i_1(P_j),\alpha^i_2(P_j))$ denote the ramification sequence of $l_i$ at $P_j$. We recall the following compatibility conditions \cite[p. 346]{EH2}:
\begin{equation}\label{lls1}
\alpha^{i-1}_j(P_i)+\alpha^i_{2-j}(P_i)\geq d-2\textrm{   for   }1\leq i\leq h+1\textrm{   and   }0\leq j\leq 2.
\end{equation}
Furthermore, any two points $Q,Q'$ on the same component $\gamma_i$ satisfy:
\begin{equation}\label{lls2}
\alpha^i_j(Q)+\alpha^i_{2-j}(Q')\leq d-2\textrm{   for   }0\leq i\leq h+1\textrm{   and   }0\leq j\leq 2.
\end{equation}
Since $E_2=\gamma_{h+1}$ is elliptic, then $\alpha^{h+1}_1(P_{h+1})\leq d-3$ and thus $\alpha^{h}_1(P_{h+1})\geq 1$ by \eqref{lls1}. Inequality \eqref{lls2} then yields $\alpha^{h}_1(P_h)\leq d-3$. By the same argument, we obtain that 
\begin{equation}\label{cusp}
\alpha^{i-1}_1(P_{i})\geq 1 \textrm{ for }1\leq i\leq h+1.
\end{equation} 
Analogously, using the fact that $E_1=\gamma_0$ is elliptic, one proves that
\begin{equation}\label{cuspbis}
\alpha^{i}_1(P_{i})\geq 1 \textrm{ for }1\leq i\leq h+1.
\end{equation}
In particular, $\gamma_0=E_1$ has at least a cusp at $P_1$ and $\gamma_{h+1}=E_2$ has at least a cusp at $P_{h+1}$.

 If $x_1$ lies on the component $\gamma_i$, then $\alpha^{i}_2(x_1)=d-2$ and thus \eqref{lls2} yields $\alpha^{i}_0(P_i)=0$ as soon as $i\neq 0$ and $\alpha^{i}_0(P_{i+1})=0$ for $i\neq h+1$. By \eqref{lls1}, we get that  both $\alpha^{i-1}_2(P_i)\geq d-2$ if $i\neq 0$ and $\alpha^{i+1}_2(P_{i+1})\geq d-2$ if $i\neq h+1$. Inductively, we obtain 
\begin{align}\label{total1}
&\alpha^j_2(P_{j+1})\geq d-2\textrm{ for }0\leq j\leq i-1\textrm{ (if $i\neq 0$) },\\
\label{total2}&\alpha^j_2(P_j)\geq d-2\textrm{ for }i+1\leq j\leq h+1\textrm{ (if $i\neq h+1$) }.
\end{align} 
We get the same conclusion if $x_2\in \gamma_i$ or $x_3\in \gamma_i$. In particular, $l_0$ has total ramification at $P_1$ as soon as at least one among the points $x_1$,$x_2$,$x_3$ does not lie on $\gamma_0=E_1$. Analogously, if at least one among $x_1$,$x_2$,$x_3$ lies outside of $\gamma_{h+1}=E_2$ we obtain that $P_{h+1}$ is a total ramification point for $l_{h+1}$.

By abuse of notation, we set $\underline{\alpha}^i(x_1)$ to be the $0$-sequence if $x_1$ does not lie on $\gamma_i$, and the same for $x_2$ and $x_3$. In the case where $x_1,x_2,x_3$ are distinct the additivity of the Brill-Noether number (cf. \cite[Lem. 3.6]{EH2}) then yields:
\begin{eqnarray*}
-4 & =&\rho(2,2,d,(0,\ldots,0,d-2),(0,\ldots,0,d-2),(0,\ldots,0,d-2)) \\
&\geq& \rho(1,2,d, \underline{\alpha}^0(P_1), \underline{\alpha}^0(x_1), 
\underline{\alpha}^0(x_2),\underline{\alpha}^0(x_3)) \\
  & & \hspace{0.5cm} + \sum_{i=1}^{h}\rho (0,2,d, 
\underline{\alpha}^i(P_{i}),\underline{\alpha}^i(P_{i+1}), 
\underline{\alpha}^i(x_1), 
\underline{\alpha}^i(x_2),\underline{\alpha}^i(x_3))\\
& & \hspace{1cm}   + \rho(1,2,d, \underline{\alpha}^{h+1}(P_{h+1}), 
\underline{\alpha}^{h+1}(x_1), 
\underline{\alpha}^{h+1}(x_2),\underline{\alpha}^{h+1}(x_3)).
\end{eqnarray*}
If $x_2=x_1$ and $x_3\neq x_1$, the above inequality still holds up to deleting all the $\underline{\alpha}^i(x_2)$. The cases where $x_3$ coincides with $x_1$ and/or $x_2$ can be treated similarly. 
We recall that:
\begin{itemize}
\item[-] the adjusted Brill-Noether number of any linear series on $\mathbb{P}^1$ with respect to any collection of points is nonnegative (cf. \cite[Thm. 1.1]{EH3});
\item[-] the adjusted Brill-Noether number of any linear series on an elliptic curve with respect to any point is nonnegative (cf. \cite[Thm. 1.1]{EH3})
\item[-] the adjusted Brill-Noether number of any $g^2_d$ on an elliptic curve with respect to any two points is $\geq -2$ (cf. \cite[Prop. 4.1]{F}).
\end{itemize}
Concerning the position of the points $x_1,x_2,x_3$, we can thus conclude (up to relabelling them) that either
\begin{itemize}
\item[(a)] $x_1$ lies  on $E_1$, $x_2$ lies on $E_2$ and $x_3$ lies on $\gamma_i$ for some $1\leq i\leq h$, or
\item[(b)] $x_1$ and $x_2$ are distinct and lie on the same elliptic component. 
\end{itemize}
In case (a), one has $\underline{\alpha}^i(x_3)\geq (0,0,d-2)$ and inequalities \eqref{cusp}, \eqref{cuspbis}, \eqref{total1}, \eqref{total2} imply both $\underline{\alpha}^i(P_i)\geq (0,1,d-2)$ and $\underline{\alpha}^i(P_{i+1})\geq (0,1,d-2)$; this contradicts the Pl\"ucker Formula \cite[Prop. 1.1]{EH1} according to which the total ramification of any $g^r_d$ on $\mathbb{P}^1$ equals $(r+1)d-r(r+1)$. 

Thus we necessarily fall in case (b). Without loss of generality, we assume that $x_1,x_2\in E_1=\gamma_0$. If $x_3= x_1$ or $x_3=x_2$, the ramification weight of $l_0$ at $x_3$ is $\geq 2(d-2)$ since it equals the sum of the weights of the ramification points of $C_b$ tending to $x_3$ (cf., e.g., \cite[p. 263]{HM}). On the other hand, $x_3$ cannot be a base point and thus $\underline{\alpha}^0(x_3)=(0,d-2,d-2)$ and this is a contradiction because $E_1$ is elliptic. We conclude that $l_0$ is totally ramified at three distinct points, namely, $x_1,x_2,x_3$ if $x_3\in E_1$ and $x_1,x_2,P_1$ if $x_3\not\in E_1$; in both cases, $l_0$ also has a cusp at $P_1$. We apply the next proposition; in the former case this implies the relation $2x_1\sim 2x_2\sim 2x_3$ on $E_1$, while in the latter case we obtain $2x_1\sim 2x_2\sim 2P_1$ on $E_1$.

Let $\pi'_0:J_{\pi'}\to B'$ be the relative generalized jacobian of the family $\pi'$, whose generic fiber is the jacobian $J(C_b)$ and whose special fiber is the generalized jacobian $J(X_0')$ parametrizing isomorphism classes of line bundles having degree $0$ on every irreducible component of $X_0'$. Hence, one has $J(X_0')\simeq  \mathrm{Pic}^0(E_1)\times\Pic^0(E_2)$ and $\pi'_0$ is a family of smooth principally polarized abelian surfaces. The relative degree-$0$ line bundle $\mathcal{O}_{\mathcal{C}'}(\sigma_2-\sigma_1)$ defines a {\em torsion} section of $\pi_0'$ (since $d\sigma_1(b)\sim d\sigma_2(b)$ by (iv)) intersecting the special fiber $J(X_0')$ in the $2$-torsion point $(\mathcal{O}_{E_1}(x_2-x_1),\mathcal{O}_{E_2})$.  By \cite[Pf. of Prop. VII.3.2 and Cor. VII.3.3]{Mi} (that works for families of abelian varieties of arbitrary dimension), the group of torsion sections of $\pi_0'$  injects in the torsion subgroup of any fiber of $\pi_0'$ and thus we conclude that $\mathcal{O}_{\mathcal{C}'}(\sigma_2-\sigma_1)$ is $2$-torsion. In particular, on the generic fiber $C_b$ of $\pi'$ the divisors $2\sigma_1(b)$ and $2\sigma_2(b)$ are linearly equivalent, that is, $\sigma_1(b)$ and $\sigma_2(b)$ are Weierstrass points.

We claim that $\sigma_3(b)$ is a Weierstrass point, as well (this is clear in the case where $2x_1\sim 2x_2\sim 2x_3$ but needs some work when $2x_1\sim 2x_2\sim 2P_1$). Let $\iota_b$ be the hyperelliptic involution on $C_b$ and set $\sigma_4(b):=\iota_b(\sigma_3(b))$. By contradiction, we assume $\sigma_4(b)\neq \sigma_3(b)$. As $d$ is even, then $d\sigma_3(b)\sim d\sigma_1(b)\sim \frac{d}{2}(\sigma_3(b)+\sigma_4(b))$ and thus $d\sigma_4(b)\sim d\sigma_3(b)\sim d\sigma_1(b)$. As a consequence, the linear series $l_b':=(\mathcal{O}_{C_b}(d\sigma_1(b)),\langle d\sigma_1(b),d\sigma_3(b),d\sigma_4(b)\rangle)$  is a $g^2_d$ on $C_b$ totally ramified at $\sigma_1(b),\sigma_3(b),\sigma_4(b)$. The first part of the proof applied to $l_b'$ thus yields that at least two points among $\sigma_1(b),\sigma_3(b),\sigma_4(b)$ are Weierstrass points of $C_b$ and thus a contradiction.

\begin{prop}\label{elliptic}
Fix an integer $d\geq 3$. If a general elliptic curve possesses a $g^2_d$ totally ramified at three points $P_1,P_2,P_3$ and with a cusp, then $d$ is even, the $g^2_d$ is not birational and the relation $2P_1\sim2P_2\sim2P_3$ holds.
\end{prop}
\begin{proof}
We first show that, if a $g^2_d$ on a general elliptic curve $E$ totally ramified at three points $P_1,P_2,P_3$ is not birational, then $d$ is even and $2P_1\sim2P_2\sim2P_3$. We consider the Stein factorization of the map $f:E\to\mathbb P^2$ defined by the $g^2_d$ (which is base point free since it admits three points of total ramification):
$$
\xymatrix{
E\ar[d]_p\ar[rd]^f\\
\mathbb P^1\ar[r]_{\hspace{-0.5cm}q}&R\subset \mathbb P^2,
}
$$
where $p$ is a cover of degree $k\geq 2$, $q$ is birational and $R$ is a singular plane curve of degree $d/k$. Since $f$ is totally ramified at $P_1,P_2,P_3$, the same holds for $p$. The Riemann-Hurwitz formula thus implies $k\leq 3$. The case $k=3$ can be excluded  for general $E$ because, by Riemann's Theorem along with the fact that all triples of points on $\mathbb P^1$ are projectively equivalent, there is a unique genus $1$ triple cover of $\mathbb P^1$ totally ramified at three points.

It remains to show that a birational $g^2_d$  on a general elliptic curve $E$ totally ramified at three points $P_1,P_2,P_3$ admits no cusps. Let $X$ be the degree $d$ plane curve image of $E$ under the map defined by the $g^2_d$. We note that the three lines $L_1$, $L_2$, $L_3$ cutting the divisors $dP_1$, $dP_2$, $dP_3$ cannot belong to a pencil of lines through a fixed point of $\mathbb P^2$ since otherwise this pencil would cut a $g^1_d$ on $E$ totally ramified at three points, thus contradicting the Riemann-Hurwitz formula for $d\geq 4$ and the generality of $E$ for $d=3$, as above.
 The curve $X$ defines a point in the {\em generalized Severi variety}
$$V_{d,1}\left(L_1+L_2+L_3, (d,d,d)\right)$$
parametrizing reduced and irreducible plane curves of geometric genus $1$ and degree $d$ having contact order $d$ at three {\em unassigned} points in the smooth locus of $L_1+L_2+L_3$. More strongly, since all triples of lines with no common points are projectively equivalent, the image of any elliptic curve under any birational $g^2_d$ totally ramified at three points is represented by a point in $V_{d,1}\left(L_1+L_2+L_3, (d,d,d)\right)$. Conversely, the normalization map of any member of $V_{d,1}\left(L_1+L_2+L_3, (d,d,d)\right)$ defines a $g^2_d$ on an elliptic curve totally ramified at three points.

We next note that an elliptic curve $E$ has a finite number of $g^2_d$ totally ramified at three points up to automorphisms; indeed, the relation $dP_1\sim dP_2\sim dP_3$ on $E$ yields that the line bundles $\mathcal O(P_i-P_j)$ are $d$-torsion elements of $\Pic^0(E)\simeq E$.   To prove the desired statement that a birational $g^2_d$  on a general elliptic curve $E$ totally ramified at three points admits no cusps, it is therefore enough to show that a general element $X$ in any irreducible component $V$ of $V_{d,1}\left(L_1+L_2+L_3, (d,d,d)\right)$ is {\em immersed} (that is, the differential of its normalization map is everywhere injective). 

Generalized Severi varieties were introduced by Caporaso-Harris in \cite{CH} (cf. \cite{Za2} for recent results on the topic). Our situation is slightly different since we fix the ra\-mi\-fication profile at three lines instead of one; however, the local computations in \cite[\S 2.2]{CH} are proved for fixed contact order with any smooth curve and thus apply also in our case where the points of contact lie in the smooth locus of $L_1+L_2+L_3$.

We recall the main deformation theoretic arguments in \cite{CH}, adapting them to\linebreak our setting\footnote{In the notation of Caporaso-Harris the line $L$ is here replaced by $L_1+L_2+L_3$ and we have $\alpha=0$, that is, we are not imposing contact order at  any fixed points of $L_1+L_2+L_3$, and $$\beta=(\underbrace{0,\ldots, 0}_\text{$d-1$},d),$$ that is, we are imposing contact order $d$ at three unassigned points of $L_1+L_2+L_3$.}. Let $X$ be a general element of any irreducible component $V$ of \linebreak\mbox{$V_{d,1}\left(L_1+L_2+L_3, (d,d,d)\right)$} and let $f:E\to X\subset \mathbb P^2$ denote the normalization map. Then $f^*(L_1+L_2+L_3)=dP_1+dP_2+dP_3$ for some points $P_1,P_2,P_3\in E$. 
 We consider the normal sheaf $N_f$, its torsion  subsheaf $H_f$ supported at the vanishing divisor $Z$ of the differential $df$ of $f$, and the quotient $\overline{N_f}:=N_f/H_f$. We have the following commutative diagram (cf. \cite[(3.51)]{Se}):
\begin{equation}\label{normal}
\xymatrix{
&&&0\ar[d]&\\
&&&H_f\ar[d]&\\
0\ar[r] &T_E\simeq\mathcal O_E\ar[d]\ar[r]^{df}& f^*T_{\mathbb P^2}\ar@{=}[d]\ar[r]& N_f\ar[d]\ar[r]& 0.\\
0\ar[r] &T_E(Z)\simeq\mathcal O_E(Z)\ar[r]& f^*T_{\mathbb P^2}\ar[r]& \overline{N_f}\ar[d]\ar[r]& 0\\
&&&0&
}
\end{equation}
As in \cite[p. 363]{CH}, for $1\leq i\leq 3$ let $l_i$ be the order of vanishing of the differential  of $f$ at $P_i$ and define the two following divisors on $E$:
\begin{align*}
D:=&\sum_{i=1}^3(d-1)P_i,\\
D_0:=&\sum_{i=1}^3l_iP_i;
\end{align*}
note that the difference $D_1:=D-D_0$ is effective in our case\footnote{and therefore coincides with the divisor $D_1$ defined in \cite{CH}.}. Furthermore, the divisor $Z-D_0$ is effective by the definition of $Z$ and $l_i$.

By \cite[Lemma 2.3 and Lemma 2.6]{CH} (cf. also \cite[p.26]{AC}), the tangent space $T_{[X]}V$ of $V$ at the point $[X]$ injects in $H^0 (E, \overline N_f(-D_1))$. From \eqref{normal} we get:
\begin{equation}\label{deg}
\mathrm{deg}\,\overline N_f(-D_1)=3d-\mathrm{deg}\,Z-\mathrm{deg}\,D+\mathrm{deg}\,D_0=3-\mathrm{deg}(Z-D_0)\leq 3,
\end{equation}
and thus
\begin{equation}\label{dim}
\dim V\leq \dim\,T_{[X]}(V)\leq h^0(E, \overline N_f(-D_1)))\leq 3.
\end{equation}
On the other hand, $3$ equals the expected dimension of $V$ because $3=3d-3(d-1)$, where $3d$ is the dimension of the Severi variety of degree $d$ genus $1$ plane curves and $3(d-1)$ comes from the ramification imposed at three unassigned points (cf. \cite[\S 2.1]{CH}). Hence, $\dim V=3$ and all inequalities in \eqref{deg} and \eqref{dim} are equalities. In par\-ti\-cular, $V$ is smooth at $[X]$ and $T_{[X]}V$ can be identified with $H^0(E,\overline N_f(-D_1))$.  Having degree $3$, the line bundle $\overline N_f(-D_1)$ is very ample and thus possesses a section vanishing at any point of $E$ with order exactly $1$; as in \cite[Proof of Prop. 2.2 p. 364]{CH}, this implies that $X$ is immersed and concludes the proof.
\end{proof}
\begin{rem}
Proposition \ref{elliptic} is sharp in the following sense. Take three points $P_1,P_2,P_3$ on an elliptic curve $E$ satisfying $2P_1\sim 2P_2\sim 2P_3$ and let $d\geq 4$ be an even integer. Then the map $f$ defined by the linear series $\langle dP_1,dP_2,dP_3\rangle$ factors as follows:
$$
\xymatrix{
E\ar[d]_p\ar[rd]^f\\
\mathbb P^1\ar[r]_{\hspace{-0.5cm}q}&R\subset \mathbb P^2,
}
$$
where $p$ is the double cover branched at $P_1,P_2,P_3$ and at a further point $P_0$, and $q$ is the (unique up to projectivities) map from $\mathbb P^1$ defined by the linear series $\langle \frac{d}{2}x_1,\frac{d}{2}x_2,\frac{d}{2}x_3\rangle$ with $x_i:=p(P_i)$. Pl\"ucker's Formula yields that $q$ has no ramification outside of $x_1,x_2,x_3$. Therefore, one computes that the ramification sequence of the $g^2_d$ on $E$ at $P_i$ is $(0,1,d-2)$ for $1\leq i\leq 3$ and $(0,1,2)$ for $i=0$. In particular, the $g^2_d$ has cusps at all the four points $P_0,P_1,P_2,P_3$.

\end{rem}


\begin{thebibliography}{HLOY}


\bibitem[AC]{AC}
E. Arbarello, M. Cornalba, {\em Su una congettura di Petri}, Comment. Math. Helv. {\bf 56} (1981), 1--38.


\bibitem[AK]{AK}
A. Altman, S. Kleiman, {\em Compactifying the Picard scheme}, Adv. Math. {\bf 35} (1980), 50--112.

\bibitem[Ba]{Ba} T. Bauer, {\it Seshadri constants on algebraic surfaces}, Math. Ann. {\bf 313} (1999), 547--583.

\bibitem[BL]{BL} 
C. Birkenhake, H. Lange, 
{\em Complex abelian varieties}, 
Second edition,  Grundlehren Math. Wiss. {\bf 302}. Springer-Verlag, Berlin, 2004.

\bibitem[BOPY]{BOPY} J. Bryan, G. Oberdieck, R. Pandharipande, Q. Yin, {\em Curve counting on abelian surfaces and threefolds}, Alg. Geom. {\bf 5} (2018), 398--463.


\bibitem[CH]{CH} L. Caporaso, J. Harris, {\em Counting plane curves of any genus}, Invent. Math. {\bf 131} (1998), 345--392.


\bibitem[Ch1]{ch1} X. Chen, {\em Rational curves on $K3$ surfaces}, J. Alg. Geom. {\bf 8} (1999), 245--278.
\bibitem[Ch2]{ch2} X. Chen, {\em A simple proof that rational curves on $K3$ are nodal}, Math. Ann. {\bf 324} (2002), 71--104.

\bibitem[De]{De} O. Debarre, {\em On the Euler characteristic of generalized Kummer varieties}, Amer. J. Math. {\bf 121} (1999), 577--586.

\bibitem[DS]{DS} T. Dedieu, E. Sernesi, {\em Equigeneric and equisingular families of curves on surfaces}, Publ. Mat. {\bf 61} (2017), 175--212.

\bibitem[DL]{DL} I. Dolgachev, D. Lehavi, {\em On isogenous principally polarized abelian surface}, in Curves and abelian varieties, AMS Contemp. Math. series {\bf 465} (2008), 51--69.




\bibitem[EH1]{EH1} 
D. Eisenbud, J. Harris, {\em Divisors on general curves and cuspidal rational curves}, Invent. Math. {\bf 74} (1983), 371--418.

\bibitem[EH2]{EH2} 
D. Eisenbud, J. Harris, {\em Limit linear series: Basic theory}, Invent. Math. {\bf 85} (1986), 337--371.

\bibitem[EH3]{EH3} 
D. Eisenbud, J. Harris, {\em The Kodaira dimension of the moduli space of curves of genus $\geq 23$}, Invent. Math. {\bf 90} (1987), 359--387.

\bibitem[F]{F}
G. Farkas, {\em The geometry of the moduli space of curves of genus $23$}, Math. Ann. {\bf 318} (2000), 43--65.


\bibitem[Go]{Go} L. G\"ottsche, {\em A conjectural generating function for numbers of curves on surfaces}, Comm. Math. Phys. {\bf 196} (1998), 523--53.



\bibitem[HKW]{HKW}
K. Hulek, C. Kahn, S. H. Weintraub, {\em Moduli spaces of abelian surfaces: Compactification, degenerations
and theta functions}, Walter de Gruyter, Berlin, 1993.

\bibitem[HM]{HM} J. Harris, I. Morrison, {\it Moduli of curves}, Graduate Texts in Math. {\bf 187}. Springer-Verlag, New York, 1991.


\bibitem[KLM]{KLM}  
A. L. Knutsen, M. Lelli-Chiesa, G. Mongardi, 
{\em Severi Varieties and Brill-Noether theory of curves on abelian surfaces}, J. Reine Angew. Math  {\bf 749} (2019), 161--200.

\bibitem[KMPS]{KMPS} A. Klemm, D. Maulik, R. Pandharipande, E. Scheidegger, {\em Noether-Lefschetz theory and the Yau-Zaslow conjecture}, J. Amer. Math. Soc. {\bf 23} (2010), 1013--1040.

\bibitem[KSS]{KSS} A. L. Knutsen, W. Syzdek, T. Szemberg, {\em Moving curves and Seshadri constants}, Math. Res. Lett. {\bf 16} (2009), 711--719.



\bibitem[LC]{LC} M. Lelli-Chiesa, {\em Curves on surfaces with trivial canonical bundle}, Boll. Unione Mat. Ital. {\bf 11} (2018), 93--105.


\bibitem[LS1]{LS1} H. Lange, E. Sernesi, {\em Curves of genus g on an abelian variety of dimension
g}, Indag. Math. {\bf 13} (2002), 523--535.

\bibitem[LS2]{LS2} H. Lange, E. Sernesi, {\em Severi varieties and branch curves of abelian surfaces
of type $(1, 3)$}, Internat. J. Math. {\bf 13} (2002), 227--244.


\bibitem[Mi]{Mi}
R. Miranda, {\em The basic theory of elliptic surfaces}, ETS Editrice, Pisa, 1989.

\bibitem[MM]{MM}
S. Mori, S. Mukai, {\em The uniruledness of the moduli space of curves of genus $11$}, Algebraic Geometry, Proc. Tokyo/Kyoto, 334--353, Lecture Notes in Math. {\bf 1016}, Springer, Berlin, 1983.

\bibitem[Mu]{Mu}
D. Mumford, {\em Abelian varieties}, Oxford University Press, London, 1970.

\bibitem[Se]{Se}
E. Sernesi, {\em Deformations of algebraic schemes}, Springer, Grundlehren der mathematischen Wissenschaften {\bf 334} (2006).


\bibitem[St]{St} A. Steffens, {\em Remarks on Seshadri constants}, Math. Z. {\bf 227} (1998), 505--510.

\bibitem[Za1]{Za} A. Zahariuc, {\em The Severi problem for abelian surfaces in the primitive case}, \texttt{arXiv:1811.01281}.


\bibitem[Za2]{Za2} A. Zahariuc, {\em The irreducibility of the generalized Severi varieties}, Proc. Lond. Math. Soc. {\bf 119}, 1431--1463.

\end{thebibliography}
\end{document}